\documentclass[10pt]{article}
\usepackage{latexsym,amsfonts,amsthm,amsmath,amssymb,cases,authblk,multicol}
\usepackage{cite}
\usepackage{enumerate}
\usepackage{graphicx}
\usepackage{epstopdf}
\usepackage{color}
\newtheorem{thm}{Theorem}[section]

\newtheorem{prob}{Problem}[section]

\newtheorem{assumption}{Assumption}[section]
\newtheorem{remark}{Remark}[section]
\newtheorem{prop}{Proposition}[section]
\newtheorem{defn}{Definition}[section]
\newtheorem{example}{Example}[section]
\newcounter{nextauthor}
\def\mathrm{\mbox}

\allowdisplaybreaks[4]
\numberwithin{remark}{section}

\begin{document}
%\begin{CJK}{GBK}{kai}
\title{{\Large \bf Stochastic Linear-quadratic Control Problems with Affine Constraints}\thanks{This work was supported by the National Natural Science Foundation of China (11671282, 12171339)  and the grant from Chongqing Technology and Business University (2356004).}}
\author[a]{Zhun Gou}
\author[b]{Nan-Jing Huang \thanks{Corresponding author: nanjinghuang@hotmail.com; njhuang@scu.edu.cn}}
\author[a]{Xian-Jun Long}
\author[c]{Jian-Hao Kang}
\affil[a]{\small\it College of Mathematics and Statistics, Chongqing Technology and Business University, Chongqing 400067, P.R. China; Chongqing Key Laboratory of Statistical Intelligent Computing and Monitoring, Chongqing Technology and Business University, Chongqing, 400067, P.R. China}
\affil[b]{Department of Mathematics, Sichuan University, Chengdu, Sichuan 610064, P.R. China}
\affil[c]{School of Mathematics, Southwest Jiaotong University, Chengdu, Sichuan 610031, P.R. China}
\date{}
\maketitle
\begin{center}
\begin{minipage}{5.6in}
\noindent{\bf Abstract.} This paper investigates the stochastic linear-quadratic control problems with affine constraints, in which both equality and inequality constraints are involved. With the help of the Pontryagin maximum principle and  Lagrangian duality theory, the dual problem of original problem is established and the state feedback form of the solution to the optimal control problem is obtained.  Under the Slater condition, the equivalence is proved between the solutions to the original problem and the ones of the dual problem, and the KKT condition is also provided for solving original problem. Especially, a new sufficient condition is given for the invertibility assumption, which ensures the uniqueness of the solutions to the dual problem.
\end{minipage}
\end{center}

\section{Introduction}
\paragraph{}

Linear-quadratic control problems (LQ problems for short) constitute the most significant part of control theory, and have been widely investigated in the literature for a variety of deterministic control systems \cite{cai2022time, li2021irregular} because of the following two reasons: (i) with the help of Riccati equations, we are able to find the feedback expressions for solutions (or optimal feedback controls) to the LQ problems, which exhibit elegant properties due to their simple and nice structures. Furthermore, the feedback expression are very useful and particularly important in many practical control problems since they keep the corresponding control to be robust with respect to (small) perturbation \cite{lu2023optimal}. (ii) in lots of nonlinear control problems, it is really hard or even impossible to find the feedback expressions, while many of them can be reasonably approximated by the LQ problems \cite{Yong1999Stochastic}.

Stochastic linear-quadratic control problems (SLQ problems for short), as important extensions of the LQ problems, can date back to the pioneered work \cite{wonham1968matrix}. Since then SLQ problems have attracted more and more researchers and remain one of the most fascinating topics in control theory \cite{cohen2018online, moon2021dynamic, wei2021infinite, tang2015dynamic, Tuffaha2023the, meng2014general}. On the other hand, for some practical reasons, one has to consider constraints on the state and control when dealing with stochastic control problems, i.e., the admissible control sets in SLQ problems with constraints are smaller than the ones (the whole space) in SLQ problems with no constraints. Various type constraints have been considered in SLQ problems in the existing works,  such as linear equality constraints \cite{zhang2023stochastic}, linear inequality constraints \cite{Wu2020On}, quadratic constraints \cite{lim1999stochastic}, terminal state constraints \cite{bank2018linear}, expectation constraints \cite{chow2020dynamic}. It is also worth mentioning that static mean-$\rho$ portfolio selection problems were investigated in the sense of both expectation constraints and risk constraints in \cite{herdegen2022mean}. Such constraints "shrink" original admissible control sets and make the stochastic control problem more complicated and more challengeable (for discrete-time LQ problems with affine constraints, we refer to \cite{mohapatra2023linear} and references therein). Nevertheless, we show by Examples \ref{24} and \ref{25} that many "shrunken" admissible control sets can be approximatively described by original admissible control sets with finite many affine-type constraints, in which the coefficients are random. Thus it would be interesting and significant to study the stochastic linear-quadratic control problems with affine constraints (SLQ-AC problems for short) in random coefficients case. However, to the best of our knowledge, there are few works studying SLQ-AC problems with random coefficients.

The overarching goal of this paper is to study the SLQ-AC problems with random coefficients, in which   equality constraints and inequality constraints are considered. There are two main contributions of this paper:
\begin{enumerate}
\item From the viewpoint of methodology, by constructing backward stochastic differential equations (BSDEs for short), we  rewrite the affine constraints as some constraints depending only on the control variable, and so the Slater condition can be more easily verified.
\item From the viewpoint of conclusions, a sufficient condition is given for the invertibility assumption, which ensures the uniqueness of the solutions to the dual problem. To the best of our knowledge, the result is new without requiring  additional controllability condition.
\end{enumerate}

The rest of this paper is structured as follows. The next section introduces some necessary preliminaries including the formulation of the SLQ problems with affine constraints and two motivating examples. In Section 3, by using the Pontryagin maximum principle, we obtain the formulation of the dual problem to which the solutions are as well the ones to original SLQ problems with affine constraints under the Slater condition, and make the invertibility assumption for ensuring the Slater condition and the uniqueness of solutions to the dual problem. Then we give the KKT condition for solving the dual problem. Before concluding this paper in Section 5, we provide a sufficient condition for guaranteeing the invertibility assumption, and give two examples to show the effectiveness of our main results.
%%%%%%%%%%%%%%%%%%%%%%%%%%%%%%%%%%%%
\section{Preliminaries}
In this section, we introduce some basic notations and present the formulation of the SLQ problems with affine constraints. Then we give two motivating examples for studying such problems.
\subsection{Problem Formulation}

Let $\mathbb{R}^{n\times m}$ be the Euclidean space of  $n\times m$-matrices $\Xi$ equipped with the norm $\|\Xi\|_{\mathbb{R}^{n\times m}}=tr^{\frac{1}{2}}(\Xi^{\top}\Xi)$ and inner product $\langle\Xi,\widetilde{\Xi}\rangle_{\mathbb{R}^{n\times m}}=tr(\Xi^{\top}\widetilde{\Xi})$ ($\widetilde{\Xi}\in \mathbb{R}^{n\times m}$), where $tr(\cdot)$ and $(\cdot)^{\top}$ represent the trace of an $n\times n$-matrix and the  transpose of an $n\times m$-matrix, respectively. Let $\mathcal{I}_n$ be the $n\times n$-identity matrix, $0^{m\times n}$ be an $m\times n$-null matrix, $\mathcal{S}^n$ the subset of symmetric matrices of $\mathbb{R}^{n\times n}$, $\mathcal{S}_{+}^n$ the subset of positive definite  matrices of $\mathcal{S}^n$, and $\mathcal{S}_{++}^n$ the subset of strictly positive definite  matrices of $\mathcal{S}_{+}^n$.

Consider the controlled linear stochastic differential equation (state equation) in a filtered probability space $(\Omega,\mathfrak{F},\mathfrak{F}_t,\mathbb{P})$ satisfying the usual hypothesis as follows:
\begin{equation}\label{SDE1}
\begin{cases}
dX(t)=\left[A(t)X(t)+B(t)u(t)\right]dt+\left[C(t)X(t)+D(t)u(t)\right]dB(t), \quad t\in[s,T],\\
X(s)=\xi\in \mathbb{R}^n.
\end{cases}
\end{equation}
where $[s,T]$ denotes a finite time duration with $T>s \geq0$; $X(\cdot)$ is the $\mathbb{R}^n$-valued state variable determined by the $\mathbb{R}^m$-valued control variable $u(\cdot)$; $B(t)$ is a standard one-dimensional Brownian motion which generates the right-continuous and increasing $\sigma$-algebra $\mathcal{F}=(\mathcal{F}_t)_{t\geq0}$; $A,C:[0,T]\times\Omega\rightarrow\mathbb{R}^{n\times n}$, $B,D\in[0,T]\times\Omega\rightarrow\mathbb{R}^{n\times m}$ are all bounded and $\mathcal{F}$-progressively measurable coefficients. For any given Euclidean space $\mathcal{H}$ and $p,q\in[1,+\infty)$, we make use of the following notations throughout this paper.
\begin{itemize}
\item $\mathcal{L}^p([s,T],\mathcal{H})$: the space of all measurable functions $f:[s,T]\rightarrow \mathcal{H}$ with $\int_s^T\|f(t)\|^p_{\mathcal{H}}dt<\infty$.

\item $\mathcal{L}^p_{\mathcal{F}_{t}}(\Omega,\mathcal{H})$: the space of all $\mathcal{F}_{t}$-measurable random variables $\eta:\Omega\rightarrow \mathcal{H}$ with $\mathbb{E}[\|\eta\|_{\mathcal{H}}^p]<\infty$.

\item $\mathcal{L}^p_{\mathcal{F}}(\Omega,C([s,T],\mathcal{H}))$: the Banach space of all $\mathcal{F}$-progressively measurable  processes $X:[s,T]\times \Omega\rightarrow \mathcal{H}$ with
$\mathbb{E}\left[\sup \limits_{t\in[s,T]}\|X(t)\|^p_{\mathcal{H}}\right]<\infty$.

\item  $\mathcal{L}^{\infty}_{\mathcal{F}}(\Omega,\mathcal{L}^p([s,T],\mathcal{H}))$: the Banach space of all $\mathcal{F}$-progressively measurable  processes $X:[s,T]\times \Omega\rightarrow \mathcal{H}$  with
$\mathop{\mbox{esssup}} \limits_{\omega\in \Omega} \left(\int_s^T\|X(t)\|_{\mathcal{H}}^pdt\right)<\infty$.

\item  $\mathcal{L}^{\infty}_{\mathcal{F}}(\Omega,C([s,T],\mathcal{H}))$: the Banach space of all $\mathcal{F}$-progressively measurable  processes $X:[s,T]\times \Omega\rightarrow \mathcal{H}$  with
$\mathop{\mbox{esssup}} \limits_{\omega\in \Omega} \left(\sup \limits_{t\in[s,T]}\|X(t)\|_{\mathcal{H}}\right)<\infty$.

\item $\mathcal{L}^q_{\mathcal{F}}(\Omega,\mathcal{L}^p([s,T],\mathcal{H}))$: the space of all $\mathcal{F}$-progressively measurable processes $X:[s,T]\times \Omega\rightarrow \mathcal{H}$  with $\mathbb{E}\left(\left[\int_s^T\|X(t)\|_{\mathcal{H}}^pdt\right]^{\frac{q}{p}}\right)<\infty$.

\end{itemize}
We denote $\mathcal{L}^p_{\mathcal{F}}(\Omega,\mathcal{L}^p([s,T],\mathcal{H}))=\mathcal{L}^p_{\mathcal{F}}([s,T],\mathcal{H})$. Clearly, $\mathcal{L}^2_{\mathcal{F}_{t}}(\Omega,\mathcal{H})$ (resp. $\mathcal{L}^2_{\mathcal{F}}([s,T],\mathcal{H})$) is a Hilbert space equipped with the norm $\|\cdot\|_{\mathcal{L}^2_{\mathcal{F}_{t}}(\Omega,\mathcal{H})}=\mathbb{E}^{\frac{1}{2}}[\|\cdot\|_{\mathcal{H}}^2]$ (resp. $\|\cdot\|_{\mathcal{L}^2_{\mathcal{F}}([s,T],\mathcal{H})}=\mathbb{E}^{\frac{1}{2}}\left[\int_s^T|\cdot(t)|_{\mathcal{H}}^2dt\right]$) and inner product $\langle\cdot,\cdot\rangle_{\mathcal{L}^2_{\mathcal{F}_{t}(\Omega,\mathcal{H})}}$ (resp. $\langle\cdot,\cdot\rangle_{\mathcal{L}^2_{\mathcal{F}}([s,T],\mathcal{H})}$). Without loss of generality, we denote $\langle\cdot,\cdot\rangle$ (resp. $|\cdot|$) the inner product (resp. the norm) in any Euclidean space.  And for notational convenience, we will frequently suppress the $t$-dependence of a process involved in an equation or an integral.

The admissible control set of $u(\cdot)$ is defined as follows, where both equality and inequality constraints are involved.
\begin{defn}
For given $l,l'\in\mathbb{R}^*$ with $1\leq l'\leq l$, and given $\alpha_i\in\mathcal{L}^2_{\mathcal{F}}([s,T],\mathbb{R}^{n})\subset\mathcal{L}^2_{\mathcal{F}}(\Omega,\mathcal{L}^1([s,T],\mathbb{R}^{n}))$, bounded $\beta_i\in\mathcal{L}^2_{\mathcal{F}}([s,T],\mathbb{R}^{m})$, and $\gamma_i\in {\mathcal{L}^{2}_{\mathcal{F}_{T}}(\Omega,\mathbb{R}^{n})}$, the admissible control set is defined as follows:
\begin{align}
U^{ad}=\Bigg\{&u\in\mathcal{L}^2_{\mathcal{F}}([s,T],\mathbb{R}^m)\big|u\;\mbox{satisfies the following affine constraints:}\nonumber\\
&\langle X,\alpha_i\rangle_{\mathcal{L}^2_{\mathcal{F}}([s,T],\mathbb{R}^n)}+\langle u,\beta_i\rangle_{\mathcal{L}^2_{\mathcal{F}}([s,T],\mathbb{R}^m)}+\langle X(T),\gamma_i\rangle_{\mathcal{L}^2_{\mathcal{F}_{T}}(\Omega,\mathbb{R}^n)}\leq a_i, \;i=1,2,\cdots,l',\nonumber\\
&\langle X,\alpha_i\rangle_{\mathcal{L}^2_{\mathcal{F}}([s,T],\mathbb{R}^n)}+\langle u,\beta_i\rangle_{\mathcal{L}^2_{\mathcal{F}}([s,T],\mathbb{R}^m)}+\langle X(T),\gamma_i\rangle_{\mathcal{L}^2_{\mathcal{F}_{T}}(\Omega,\mathbb{R}^n)}= a_i, \;i=l'+1,l'+2,\cdots,l\Bigg\}\label{3},
\end{align}
where $a=(a_1,a_2,\cdots,a_l)^{\top}\in \mathbb{R}^l$.
\end{defn}
For any $u\in U^{ad}$, the state equation \eqref{SDE1} admits a unique solution $X\in \mathcal{L}^2_{\mathcal{F}}(\Omega,C([s,T],\mathbb{R}^n))$ (see \cite{lu2021mathematical}). Moreover, letting $(\widetilde{R}_i,\widetilde{r}_i)\in\mathcal{L}^2_{\mathcal{F}}(\Omega,C([s,T],\mathbb{R}^n))\times\mathcal{L}^2_{\mathcal{F}}([s,T],\mathbb{R}^n)$ be the unique solution to the following BSDE (\cite[Proposition 3.2]{yong2020stochastic})
\begin{equation*}
\begin{cases}
d\widetilde{R}_i=-\left[A^{\top}\widetilde{R}_i+C^{\top}\widetilde{r}_i+\alpha_i\right]dt+\widetilde{r}_i dB(t), \quad t\in[s,T],\\
\widetilde{R}_i(T)=\gamma_i
\end{cases}
\end{equation*}
and applying the It\^{o} formula to $\langle X,\widetilde{R}_i\rangle$, the admissible control set $U^{ad}$  given by \eqref{3} can be rewritten as
\begin{align}\label{100}
U^{ad}=\Bigg\{u\in\mathcal{L}^2_{\mathcal{F}}([s,T],\mathbb{R}^m)\bigg|&\left\langle u,\widetilde{\rho}_i\right\rangle_{\mathcal{L}^2_{\mathcal{F}}([s,T],\mathbb{R}^m)}\leq \widetilde{a}_i,\; i=1,2,\cdots,l',\nonumber\\
&\left\langle u,\widetilde{\rho}_i\right\rangle_{\mathcal{L}^2_{\mathcal{F}}([s,T],\mathbb{R}^m)}= \widetilde{a}_i,\; i=l'+1,l'+2,\cdots,l\Bigg\}
\end{align}
with $\widetilde{\rho}_i=\beta_i+B^{\top}\widetilde{R}_i+D^{\top}\widetilde{r}_i$ and $\widetilde{a}_i=a_i-\langle \xi, \mathbb{E}[R_i(0)]\rangle$. \eqref{100} indeed reveals the essence of the affine constraints, since it is independent of the state variable $X$. Obviously, $U^{ad}$ is closed and convex in $\mathcal{L}^2_{\mathcal{F}}([s,T],\mathbb{R}^m)$.

We further make the following assumption for \eqref{100}, which seems not very hard to be verified.
\begin{assumption}\label{17}
(the Slater condition) There exists at least one $\overline{u}\in U^{ad}$ such that
\begin{equation}\label{99}
\begin{cases}
\left\langle u,\widetilde{\rho}_i\right\rangle_{\mathcal{L}^2_{\mathcal{F}}([s,T],\mathbb{R}^m)}< \widetilde{a}_i, \;i=1,2,\cdots,l',\\
\left\langle u,\widetilde{\rho}_i\right\rangle_{\mathcal{L}^2_{\mathcal{F}}([s,T],\mathbb{R}^m)}= \widetilde{a}_i, \;i=l'+1,l'+2,\cdots,l,
\end{cases}
\end{equation}
where the controlled pair $(\overline{X},\overline{u})$ is governed by the state equation \eqref{SDE1}.
\end{assumption}
\begin{remark}\label{98}
Without loss of ambiguity, the vector group $\{\widetilde{\rho}_i\}_{i=l'+1}^l$ of equality constraints is required to be linearly independent, i.e., if there exists a vector $k=(k_{l
+1},k_{l'+2},\cdots,k_l)^{\top}\in\mathbb{R}^{l-l'}$ such that $\sum \limits_{i=l'+1}^{l}k_iD_i=0^{n\times 1}$, $\forall t\in[s,T],\;\mathbb{P}$-a.s., then $k=0^{(l-l')\times 1}$. Furthermore, let $\overline{\rho}=(\widetilde{\rho}_{l'+1},\widetilde{\rho}_{l'+2},\cdots,\widetilde{\rho}_l)$ and $\overline{a}=(\widetilde{a}_{l'+1},\widetilde{a}_{l'+2},\cdots,\widetilde{a}_l)^{\top}$, and consider the map $\Xi:\mathcal{L}^2_{\mathcal{F}}([s,T],\mathbb{R}^m)\rightarrow \mathbb{R}^{l-l'},\;\Xi(u)=\mathbb{E}[\overline{\rho}^{\top}u]-\overline{a}$. Then, by choosing $u=\overline{\rho}\theta\;(\forall\theta\in\mathbb{R}^{l-l'})$,  we can see $\mathbb{E}[\overline{\rho}^{\top}\overline{\rho}]$ is invertible and so $\Xi$ is surjective, which coincides with assumption (A4) in \cite{zhang2023stochastic} and can be satisfied under  some Kalman-type rank condition.
\end{remark}
For the controlled pair $(X,u)$ governed by the state equation \eqref{SDE1}, consider the following cost functional
\begin{equation}\label{8}
J(s,\xi,u)=\mathbb{E}\left[\displaystyle{\int_s^T}\left\langle
\begin{bmatrix}
E(t)& F(t)&G(t)\\
F^{\top}(t)& I(t)&K(t)\\
G^{\top}(t)& K^{\top}(t)&0
\end{bmatrix}
\begin{bmatrix}
X(t)\\
u(t)\\
1
\end{bmatrix},
\begin{bmatrix}
X(t)\\
u(t)\\
1
\end{bmatrix}\right\rangle
dt+\left\langle
\begin{bmatrix}
M& N\\
N^{\top}& 0
\end{bmatrix}
\begin{bmatrix}
X(T)\\
1
\end{bmatrix},
\begin{bmatrix}
X(T)\\
1
\end{bmatrix}\right\rangle
\right],
\end{equation}
where $G\in\mathcal{L}^2_{\mathcal{F}}([s,T],\mathbb{R}^{n})\subset\mathcal{L}^2_{\mathcal{F}}(\Omega,\mathcal{L}^1([s,T],\mathbb{R}^{n}))$,  $K\in\mathcal{L}^2_{\mathcal{F}}([s,T],\mathbb{R}^{m})$,
$N\in\mathcal{L}^{2}_{\mathcal{F}_{T}}(\Omega,\mathbb{R}^{n})$ are all bounded. Moreover,
$E:[0,T]\times\Omega\rightarrow\mathcal{S}_{+}^n$,
$F:[0,T]\times\Omega\rightarrow\mathbb{R}^{n\times m}$, $I:[0,T]\times\Omega\rightarrow\mathcal{S}_{++}^m$ and $M:[0,T]\times\Omega\rightarrow\mathcal{S}_{+}^n$ are all bounded and $\mathcal{F}$-progressively measurable processes satisfying the following assumption.
\begin{assumption}\label{convexity}
For any $\varepsilon>0$, there is
\begin{equation}\label{123}
\mathbb{E}\left[\displaystyle{\int_s^T}\left\langle
\begin{bmatrix}
E(t)& F(t)\\
F^{\top}(t)& I(t)
\end{bmatrix}
\begin{bmatrix}
X(t)\\
u(t)
\end{bmatrix},
\begin{bmatrix}
X(t)\\
u(t)
\end{bmatrix}\right\rangle
dt+\left\langle MX(T),X(T)\right\rangle
\right]\geq \varepsilon\|u\|_{\mathcal{L}^2_{\mathcal{F}}([s,T],\mathbb{R}^m)}.
\end{equation}
\end{assumption}
\begin{remark}
\eqref{123} indicates the convexity of $J(s,\xi,u)$ with respect to $u$.
\end{remark}
In the present paper, we would like to study the SLQ-AC problem specified as follows.
\begin{prob}\label{2}
Find $\widehat{u}\in U^{ad}$ such that
$$
J(s,\xi,\widehat{u})=\inf\limits_{u \in U^{ad}}J(s,\xi,u),
$$
where $J(s,\xi,u)$ is given by \eqref{8}.
\end{prob}

\begin{remark}
There are two features for Problem \ref{2}: (i) in the state equation, affine constraints and the cost functional, all the coefficients are allowed to be random; (ii) both the drift term and diffusion term of the state equation depend on the state variable and the control variable. Therefore we are investigating the SLQ-AC problem  in a quite general framework.
\end{remark}
In the sequel, we call $J(s,\xi,\widehat{u})$ the value function of Problem \ref{2}, which is related to $(s,\xi,a)\in [0,T)\times \mathbb{R}^n\times \mathbb{R}^l$.

\subsection{Motivating Examples}
In this section, we give two  examples to show that some closed and convex subsets in $\mathcal{L}^2_{\mathcal{F}}([s,T],\mathbb{R}^m)$ can be approximately described by adding finite many affine constraints on the space $\mathcal{L}^2_{\mathcal{F}}([s,T],\mathbb{R}^m)$. The first example is inspired by \cite{lim1999stochastic}, which is concerned with the quadratic constraints.
\begin{example}\label{24}
For bounded and $\mathcal{F}$-progressively measurable functionals $\widetilde{E}:[0,T]\times\Omega\rightarrow\mathcal{S}_{+}^n$,
$\widetilde{I}:[0,T]\times\Omega\rightarrow\mathcal{S}_{++}^m$, and $\widetilde{M}:[0,T]\times\Omega\rightarrow\mathcal{S}_{+}^n$, consider the quadratic functional $\Theta: U_1\rightarrow \mathbb{R}$ defined by
$$
\Theta(u)=\langle X,\widetilde{E}X\rangle_{\mathcal{L}^2_{\mathcal{F}}([s,T],\mathbb{R}^n)}+\langle u,\widetilde{I}u\rangle_{\mathcal{L}^2_{\mathcal{F}}([s,T],\mathbb{R}^m)}+\langle X(T),\widetilde{M}X(T)\rangle_{\mathcal{L}^2_{\mathcal{F}_{T}}(\Omega,\mathbb{R}^n)},
$$
where
$
U_1=\left\{u\in\mathcal{L}^2_{\mathcal{F}}([s,T],\mathbb{R}^m)\big|\Theta(u)\leq a_0\right\}
$
and the controlled pair $(X,u)$ is determined by the state equation \eqref{SDE1}, and where $a_0$ is a given positive constant. Then one can easily show that $U_1$ is closed and convex in $\mathcal{L}^2_{\mathcal{F}}([s,T],\mathbb{R}^m)$. On the other hand, let
\begin{align*}
U_2=\Bigg\{u\in\mathcal{L}^2_{\mathcal{F}}([s,T],\mathbb{R}^m)\big|&\langle X,\widetilde{E}\widetilde{\alpha}\rangle_{\mathcal{L}^2_{\mathcal{F}}([s,T],\mathbb{R}^n)}+\langle u,\widetilde{I}\widetilde{\beta}\rangle_{\mathcal{L}^2_{\mathcal{F}}([s,T],\mathbb{R}^m)}+\langle X(T),\widetilde{M}\widetilde{\gamma}\rangle_{\mathcal{L}^2_{\mathcal{F}_{T}}(\Omega,\mathbb{R}^n)}\leq a_0,\\
&\mbox{for all}\;(\widetilde{\alpha},\widetilde{\beta},\widetilde{\gamma})\in\mathcal{L}^2_{\mathcal{F}}([s,T],\mathbb{R}^n)
\times\mathcal{L}^2_{\mathcal{F}}([s,T],\mathbb{R}^m)\times\mathcal{L}^2_{\mathcal{F}_{T}}(\Omega,\mathbb{R}^n)\;\mbox{such that}\\
&\langle\widetilde{E}\widetilde{\alpha},\widetilde{\alpha}\rangle_{\mathcal{L}^2_{\mathcal{F}}([s,T],\mathbb{R}^n)}+\langle \widetilde{I}\widetilde{\beta},\widetilde{\beta}\rangle_{\mathcal{L}^2_{\mathcal{F}}([s,T],\mathbb{R}^m)}+\langle\widetilde{M}\widetilde{\gamma},\widetilde{\gamma}\rangle_{\mathcal{L}^2_{\mathcal{F}_{T}}(\Omega,\mathbb{R}^n)}= a_0\Bigg\}.
\end{align*}
Then for any $(\widetilde{\alpha},\widetilde{\beta},\widetilde{\gamma})\in\mathcal{L}^2_{\mathcal{F}}([s,T],\mathbb{R}^n)
\times\mathcal{L}^2_{\mathcal{F}}([s,T],\mathbb{R}^m)\times\mathcal{L}^2_{\mathcal{F}_{T}}(\Omega,\mathbb{R}^n)$ satisfying $$\langle\widetilde{E}\widetilde{\alpha},\widetilde{\alpha}\rangle_{\mathcal{L}^2_{\mathcal{F}}([s,T],\mathbb{R}^n)}+\langle \widetilde{I}\widetilde{\beta},\widetilde{\beta}\rangle_{\mathcal{L}^2_{\mathcal{F}}([s,T],\mathbb{R}^m)}+\langle\widetilde{M}\widetilde{\gamma},\widetilde{\gamma}\rangle_{\mathcal{L}^2_{\mathcal{F}_{T}}(\Omega,\mathbb{R}^n)}= a_0,$$ we have $U_1\subseteq U_2$ because
\begin{align*}
0\leq &\langle X-\widetilde{\alpha},\widetilde{E}(X-\widetilde{\alpha})\rangle_{\mathcal{L}^2_{\mathcal{F}}([s,T],\mathbb{R}^n)}+\langle u-\widetilde{\beta},\widetilde{I}(u-\widetilde{\beta})\rangle_{\mathcal{L}^2_{\mathcal{F}}([s,T],\mathbb{R}^m)}+\langle X(T)-\widetilde{\gamma},\widetilde{M}(X(T)-\widetilde{\gamma})\rangle_{\mathcal{L}^2_{\mathcal{F}_{T}}(\Omega,\mathbb{R}^n)}\\
\leq &\Theta(u)+\left[\langle\widetilde{E}\widetilde{\alpha},\widetilde{\alpha}\rangle_{\mathcal{L}^2_{\mathcal{F}}([s,T],\mathbb{R}^n)}+\langle \widetilde{I}\widetilde{\beta},\widetilde{\beta}\rangle_{\mathcal{L}^2_{\mathcal{F}}([s,T],\mathbb{R}^m)}+\langle\widetilde{M}\widetilde{\gamma},\widetilde{\gamma}\rangle_{\mathcal{L}^2_{\mathcal{F}_{T}}(\Omega,\mathbb{R}^n)}\right]\\
&\mbox{}-2\left[\langle X,\widetilde{E}\widetilde{\alpha}\rangle_{\mathcal{L}^2_{\mathcal{F}}([s,T],\mathbb{R}^n)}+\langle u,\widetilde{I}\widetilde{\beta}\rangle_{\mathcal{L}^2_{\mathcal{F}}([s,T],\mathbb{R}^m)}+\langle X(T),\widetilde{M}\widetilde{\gamma}\rangle_{\mathcal{L}^2_{\mathcal{F}_{T}}(\Omega,\mathbb{R}^n)}\right]\\
\leq &2a_0-2\left[\langle X,\widetilde{E}\widetilde{\alpha}\rangle_{\mathcal{L}^2_{\mathcal{F}}([s,T],\mathbb{R}^n)}+\langle u,\widetilde{I}\widetilde{\beta}\rangle_{\mathcal{L}^2_{\mathcal{F}}([s,T],\mathbb{R}^m)}+\langle X(T),\widetilde{M}\widetilde{\gamma}\rangle_{\mathcal{L}^2_{\mathcal{F}_{T}}(\Omega,\mathbb{R}^n)}\right].
\end{align*}
Moreover, by choosing $\widetilde{\alpha}=\sqrt{\frac{a_0}{\Theta(u)}}X$, $\widetilde{\beta}=\sqrt{\frac{a_0}{\Theta(u)}}u$ and $\widetilde{\gamma}=\sqrt{\frac{a_0}{\Theta(u)}}X(T)$, one has $U_2\subseteq U_1$. Hence we conclude that $U_1=U_2$. Let
\begin{align*}
U^{p}_2=\Bigg\{u\in&\mathcal{L}^2_{\mathcal{F}}([s,T],\mathbb{R}^m)\big|\langle X,\widetilde{E}\widetilde{\alpha}^i\rangle_{\mathcal{L}^2_{\mathcal{F}}([s,T],\mathbb{R}^n)}+\langle u,\widetilde{I}\widetilde{\beta}^i\rangle_{\mathcal{L}^2_{\mathcal{F}}([s,T],\mathbb{R}^m)}+\langle X(T),\widetilde{M}\widetilde{\gamma}^i\rangle_{\mathcal{L}^2_{\mathcal{F}_{T}}(\Omega,\mathbb{R}^n)}\leq a_0,\\
&\mbox{for all}\;(\widetilde{\alpha}^i,\widetilde{\beta}^i,\widetilde{\gamma}^i)\in\mathcal{L}^2_{\mathcal{F}}([s,T],\mathbb{R}^n)
\times\mathcal{L}^2_{\mathcal{F}}([s,T],\mathbb{R}^m)\times\mathcal{L}^2_{\mathcal{F}_{T}}(\Omega,\mathbb{R}^n)\;(i=1,2,\cdots,p)\\
&\mbox{such that}\;\mathbb{P}\left\{(\widetilde{\alpha}^i,\widetilde{\beta}^i,\widetilde{\gamma}^i)=(\widetilde{\alpha}^j,\widetilde{\beta}^j,\widetilde{\gamma}^j),\;i\neq j\right\}=0\;\mbox{and}\\
&\langle\widetilde{E}\widetilde{\alpha}^i,\widetilde{\alpha}^i\rangle_{\mathcal{L}^2_{\mathcal{F}}([s,T],\mathbb{R}^n)}+\langle \widetilde{I}\widetilde{\beta}^i,\widetilde{\beta}^i\rangle_{\mathcal{L}^2_{\mathcal{F}}([s,T],\mathbb{R}^m)}+\langle\widetilde{M}\widetilde{\gamma}^i,\widetilde{\gamma}^i\rangle_{\mathcal{L}^2_{\mathcal{F}_{T}}(\Omega,\mathbb{R}^n)}= a_0\Bigg\}.
\end{align*}
Then $U_1$ can be approximately described by $U^{p}_2$ when $p\in \mathbb{N}^*$ is large enough.
\end{example}
The second example is concerned with the risk constraints (see, e.g., \cite{herdegen2022mean}).
\begin{example}\label{25}
For a constant $a_0>0$, let $$U_3=\left\{u\in\mathcal{L}^2_{\mathcal{F}}([s,T],\mathbb{R}^m)\big|\mbox{Var}(X(T))=\mathbb{E}(X^2(T))-\mathbb{E}^2(X(T))\leq a_0\right\}$$ and
\begin{align*}
U_4=\Bigg\{u\in\mathcal{L}^2_{\mathcal{F}}([s,T],\mathbb{R}^m)\big|&\left\langle \widehat{\gamma},X(T)\right\rangle_{\mathcal{L}^2_{\mathcal{F}_{T}}(\Omega,\mathbb{R}^n)}\leq a_0\\
&\mbox{for all}\;\widehat{\gamma}\in\mathcal{L}^2_{\mathcal{F}_{T}}(\Omega,\mathbb{R}^n)\;\mbox{such that}\; \mathbb{E}(\widehat{\gamma})=0\;\mbox{and}\;\mbox{Var}(\widehat{\gamma})=a_0\Bigg\},
\end{align*}
where $\mbox{Var}(\cdot)$ denotes the variance of a random variable, and the controlled pair $(X,u)$ is governed by the state equation \eqref{SDE1}. By similar arguments in Example \ref{24}, we can show that $U_3$ is convex and closed in $\mathcal{L}^2_{\mathcal{F}}([s,T],\mathbb{R}^m)$ and $U_3=U_4$. Let
\begin{align*}
U^{p}_4=\Bigg\{u\in\mathcal{L}^2_{\mathcal{F}}([s,T],\mathbb{R}^m)\big|&\left\langle \widehat{\gamma}^i,X(T)\right\rangle_{\mathcal{L}^2_{\mathcal{F}_{T}}(\Omega,\mathbb{R}^n)}\leq a_0,\quad\mbox{for all}\;\widehat{\gamma}^i\in\mathcal{L}^2_{\mathcal{F}_{T}}(\Omega,\mathbb{R}^n)\;(i=1,2,\cdots,p)\\
&\mbox{such that} \;\mathbb{P}\left\{\widetilde{\gamma}^i=\widetilde{\gamma}^j,\;i\neq j\right\}=0,\;\mathbb{E}(\widehat{\gamma}^i)=0\;\mbox{and}\;\mbox{Var}(\widehat{\gamma}^i)=a_0\Bigg\}.
\end{align*}
Then $U_3$ can be approximately described by $U^{p}_4$ when $p\in \mathbb{N}^*$ is large enough.
\end{example}

\section{Dual Problem and the KKT condition}
In this section, we give the formulation of the dual problem for Problem \ref{2}, and then derive the KKT condition for solving the dual problem.
In order to solve Problem \ref{2}, we introduce the following relaxed SLQ problem.
\begin{prob}\label{13}
For any given $\lambda=(\lambda_1,\lambda_2,\cdots,\lambda_l)^{\top}$ with $\lambda_i\geq0$ ($i=1,2,\cdots,l'$) and $\lambda_i\in \mathbb{R}$ ($i=l'+1,l'+2,\cdots,l$), find $u^{*}=u^*(\lambda)$ such that
\begin{equation}\label{+10}
\widetilde{J}(s,\xi,\lambda,u^*(\lambda))=\inf\limits_{u \in \mathcal{L}^2_{\mathcal{F}}([s,T],\mathbb{R}^{n})}\widetilde{J}(s,\xi,\lambda,u),
\end{equation}
where
$$
\widetilde{J}(s,\xi,\lambda,u)=J(u)+2\sum \limits_{i=1}^{l}\lambda_i\left[\langle X,\alpha_i\rangle_{\mathcal{L}^2_{\mathcal{F}}([s,T],\mathbb{R}^n)}+\langle u,\beta_i\rangle_{\mathcal{L}^2_{\mathcal{F}}([s,T],\mathbb{R}^m)}+\langle X(T),\gamma_i\rangle_{\mathcal{L}^2_{\mathcal{F}_{T}}(\Omega,\mathbb{R}^n)}\right]-2\langle a,\lambda\rangle.
$$
\end{prob}
\begin{remark}
$\widetilde{J}(s,\xi,\lambda,u^*(\lambda))$ is nothing but the Lagrangian dual functional of Problem \ref{2}.
\end{remark}
For Problem \ref{13},  the Hamiltonian functional can be given as follows:
$$
H=-\frac{1}{2}\left\langle
\begin{bmatrix}
E& F&G\\
F^{\top}& I&K\\
G^{\top}& K^{\top}&0
\end{bmatrix}
\begin{bmatrix}
X\\
u\\
1
\end{bmatrix},
\begin{bmatrix}
X\\
u\\
1
\end{bmatrix}\right\rangle-\sum \limits_{i=1}^{l}\lambda_i\left[\langle X,\alpha_i\rangle+\langle u,\beta_i\rangle\right]+\langle AX+Bu,Y\rangle
+\langle CX+Du,Z\rangle.
$$
Clearly, the derivatives of $H$ are
\begin{equation*}
\begin{cases}
\frac{\partial H}{\partial X}=-\left(EX+Fu+G\right)-\sum \limits_{i=1}^{l}\lambda_i\alpha_i+A^{\top}Y+C^{\top}Z,\\
\frac{\partial H}{\partial u}=-\left(F^{\top}X+Iu+K\right)-\sum \limits_{i=1}^{l}\lambda_i\beta_i+B^{\top}Y+D^{\top}Z
\end{cases}
\end{equation*}
and the adjoint BSDE is
\begin{equation*}
\begin{cases}
dY=\left[EX+Fu+G+\sum \limits_{i=1}^{l}\lambda_i\alpha_i-A^{\top}Y-C^{\top}Z\right]dt+ZdB(t), \quad t\in[s,T],\\
Y(T)=-MX(T)-N-\sum \limits_{i=1}^{l}\lambda_i\gamma_i.
\end{cases}
\end{equation*}
According to the Pontryagin maximum principle (see, e.g.,  \cite[Theorem 3.3 in  Chapter 3]{Yong1999Stochastic}), the optimal control $u^*$ of Problem \ref{13} can be given as follows:
\begin{equation}\label{S1}
Iu^*=-\left(F^{\top}X^*+K\right)-\sum \limits_{i=1}^{l}\lambda_i\beta_i+B^{\top}Y^*+D^{\top}Z^*,
\end{equation}
where $(X^*,u^*,Y^*,Z^*)=(X^*(\lambda),u^*(\lambda),Y^*(\lambda),Z^*(\lambda))$ satisfies the following fully decoupled forward-backward stochastic differential equation (FBSDE):
\begin{equation}\label{+8}
\begin{cases}
dX^*=\left[AX^*+Bu^*\right]dt+\left[CX^*+Du^*\right]dB(t), \quad t\in[s,T],\\
dY^*=\left[EX^*+Fu^*+G+\sum \limits_{i=1}^{l}\lambda_i\alpha_i-A^{\top}Y^*-C^{\top}Z^*\right]dt+Z^*dB(t), \quad t\in[s,T],\\
X^*(s)=\xi\in \mathbb{R}^n,\quad Y^*(T)=-MX^*(T)-N-\sum \limits_{i=1}^{l}\lambda_i\gamma_i.
\end{cases}
\end{equation}
By \cite[Theorem 2.2]{yong2006linear}, the FBSDE \eqref{+8} admits a unique solution $(X^*,u^*,Y^*,Z^*)\in \mathcal{L}^2_{\mathcal{F}}(\Omega,C([s,T],\mathbb{R}^n))\times U^{ad}\times \mathcal{L}^2_{\mathcal{F}}(\Omega,C([s,T],\mathbb{R}^n))\times\mathcal{L}^2_{\mathcal{F}}([s,T],\mathbb{R}^n))$ for any given $\lambda\in\mathbb{R}^l$. It remains to be verified that $u^*=u^*(\lambda)$ given by \eqref{S1} satisfies condition \eqref{+10}. For the optimal pair $(X^*,u^*)$ and any given control pair $(X,u)\in \mathcal{L}^2_{\mathcal{F}}([s,T],\mathbb{R}^{n})\times \mathcal{L}^2_{\mathcal{F}}([s,T],\mathbb{R}^{m})$, letting $\Delta X=X^*-X$ and $\Delta u=u^*-u$, one has $\Delta X(s)=0$ and
\begin{align}\label{+5}
&\quad\,\frac{1}{2}\left[\widetilde{J}(s,\xi,\lambda,u^*(\lambda))-\widetilde{J}(s,\xi,\lambda,u)\right]\nonumber\\
&\leq \left\langle
\begin{bmatrix}
E& F&G\\
F^{\top}& I&K\\
G^{\top}& K^{\top}&0
\end{bmatrix}
\begin{bmatrix}
\Delta X\\
\Delta u\\
0
\end{bmatrix},
\begin{bmatrix}
X^*\\
u^*\\
1
\end{bmatrix}\right\rangle_{\mathcal{L}^2_{\mathcal{F}}([s,T],\mathbb{R}^{n+m+1})}
+\left\langle
\begin{bmatrix}
M& N\\
N^{\top}& 0
\end{bmatrix}
\begin{bmatrix}
\Delta X(T)\\
0
\end{bmatrix},
\begin{bmatrix}
X^*(T)\\
1
\end{bmatrix}\right\rangle_{\mathcal{L}^2_{\mathcal{F}_{T}}(\Omega,\mathbb{R}^{n+1})}\nonumber\\
&\mbox{}+\sum \limits_{i=1}^{l}\lambda_i\left[\langle \Delta X,\alpha_i\rangle_{\mathcal{L}^2_{\mathcal{F}}([s,T],\mathbb{R}^n)}+\langle \Delta u,\beta_i\rangle_{\mathcal{L}^2_{\mathcal{F}}([s,T],\mathbb{R}^m)}+\langle\Delta X(T),\gamma_i\rangle_{\mathcal{L}^2_{\mathcal{F}_{T}}(\Omega,\mathbb{R}^n)}\right].
\end{align}
Noticing that
\begin{equation}\label{+7}
\left\langle
\begin{bmatrix}
M& N\\
N^{\top}& 0
\end{bmatrix}
\begin{bmatrix}
\Delta X(T)\\
0
\end{bmatrix},
\begin{bmatrix}
X^*(T)\\
1
\end{bmatrix}\right\rangle_{\mathcal{L}^2_{\mathcal{F}_{T}}(\Omega,\mathbb{R}^{n+1})}+\langle\Delta X(T),\gamma_i\rangle_{\mathcal{L}^2_{\mathcal{F}_{T}}(\Omega,\mathbb{R}^n)}=-\left\langle \Delta X(T),Y^*(T)\right\rangle_{\mathcal{L}^2_{\mathcal{F}_{T}}(\Omega,\mathbb{R}^n)},
\end{equation}
applying the It\^{o} formula to $\langle \Delta X,Y^*\rangle$ derives
\begin{align}\label{+6}
&\quad\,\langle\Delta X(T),Y^*\rangle_{\mathcal{L}^2_{\mathcal{F}_{T}}(\Omega,\mathbb{R}^n)}\nonumber\\
&=\langle \Delta u,B^{\top}Y^*+D^{\top}Z^*\rangle_{\mathcal{L}^2_{\mathcal{F}}([s,T],\mathbb{R}^m)}+\langle EX^*+Fu^*+G+\sum \limits_{i=1}^{l}\lambda_i\alpha_i,\Delta X\rangle_{\mathcal{L}^2_{\mathcal{F}}([s,T],\mathbb{R}^n)}\nonumber\\
&=\langle \Delta u,Iu^*+F^{\top}X^*+K+\sum \limits_{i=1}^{l}\lambda_i\beta_i\rangle_{\mathcal{L}^2_{\mathcal{F}}([s,T],\mathbb{R}^m)}+\langle EX^*+Fu^*+G+\sum \limits_{i=1}^{l}\lambda_i\alpha_i,\Delta X\rangle_{\mathcal{L}^2_{\mathcal{F}}([s,T],\mathbb{R}^n)}.
\end{align}
By \eqref{+5}, \eqref{+7} and \eqref{+6}, we have $\widetilde{J}(s,\xi,\lambda,u^*(\lambda))-\widetilde{J}(s,\xi,\lambda,u)\leq0$, which implies that $u^*=u^*(\lambda)$ given by \eqref{S1} is an optimal control of Problem \ref{13}.

Next, we search for the state feedback form of $u^*$. Observing the linearity of \eqref{+8} and following \cite{tang2003general}, we assume that
\begin{equation}\label{+16}
Y^*=-PX^*-Q-\sum \limits_{i=1}^{l}\lambda_iR_i.
\end{equation}
 Then one can show that $P$ is described by the following backward stochastic Riccati equation (BSRE):
\begin{equation}\label{4}
\begin{cases}
dP=-\left[E+PA+A^{\top}P+C^{\top}\Lambda+\Lambda C+C^{\top}PC-L^{\top}S^{-1}L\right]dt+\Lambda dB(t), \quad t\in[s,T],\\
P(T)=M,
\end{cases}
\end{equation}
and $Q$, $R_i$ satisfy the following system of BSDEs:
\begin{equation}\label{15}
\begin{cases}
dQ=-\left[A^{\top}Q+C^{\top}\pi+G-L^{\top}S^{-1}\psi\right]dt+\pi dB(t), \quad t\in[s,T],\\
dR_i=-\left[A^{\top}R_i+C^{\top}r_i+\alpha_i-L^{\top}S^{-1}\rho_i\right]dt+r_i dB(t), \quad t\in[s,T],\\
Q(T)=N,\quad R_i(T)=\gamma_i.
\end{cases}
\end{equation}
Here  $S=I+D^{\top}PD$, $L=F^{\top}+B^{\top}P+D^{\top}\Lambda+D^{\top}PC$, $\psi=K+B^{\top}Q+D^{\top}\pi$ and $\rho_i=\beta_i+B^{\top}R_i+D^{\top}r_i$.
The existence and uniqueness of solutions to \eqref{4} and \eqref{15} can be ensured by the following proposition.
\begin{prop}\label{27}
Under Assumption \ref{convexity}, the following two assertions
hold:
\begin{enumerate}[(I)]
\item There exists a unique solution $(P,\Lambda)\in\mathcal{L}^{\infty}_{\mathcal{F}}(\Omega,C([s,T],\mathcal{S}^n))\times\mathcal{L}^2_{\mathcal{F}}([s,T],\mathcal{S}^n)$ to BSRE \eqref{4} such that for some $\varepsilon>0$, the inequality
\begin{equation}\label{16}
S=I+D^{\top}PD\geq \varepsilon \mathcal{I}_m,\qquad \mbox{a.e.} \;t\in[s,T], \;\mathbb{P}-\mbox{a.s},
\end{equation}
holds, i.e., $\left\langle(I+D^{\top}PD)v,v\right\rangle_{\mathcal{L}^2_{\mathcal{F}}([s,T],\mathbb{R}^m)}\geq\varepsilon\|v\|^2_{\mathcal{L}^2_{\mathcal{F}}([s,T],\mathbb{R}^m)}$ for all $v\in\mathcal{L}^2_{\mathcal{F}}([s,T],\mathbb{R}^m)$.

\item There exists a unique solution $$(Q,\pi,R_i,r_i)\in\mathcal{L}^2_{\mathcal{F}}(\Omega,C([s,T],\mathbb{R}^n))\times\mathcal{L}^2_{\mathcal{F}}([s,T],\mathbb{R}^n)
    \times\mathcal{L}^2_{\mathcal{F}}(\Omega,C([s,T],\mathbb{R}^n))\times\mathcal{L}^2_{\mathcal{F}}([s,T],\mathbb{R}^n)$$ to the system of BSDEs  \eqref{15}.
\end{enumerate}
\end{prop}
\begin{proof}
(I) follows from \cite[Theorem 6.3]{sun2021indefinite} (or \cite[Theorems 5.1 and 5.2]{zhang2020backward}) directly. Then, noticing that $\beta_i,K\in\mathcal{L}^2_{\mathcal{F}}([s,T],\mathbb{R}^{m})$ are bounded, we know $L^{\top}S^{-1}\beta_i,L^{\top}S^{-1}K\in \mathcal{L}^2_{\mathcal{F}}(\Omega,\mathcal{L}^1([s,T],\mathbb{R}^{n}))$. And so (II) follows from \cite[Proposition 3.2]{yong2020stochastic} immediately.
\end{proof}
Combining \eqref{S1}-\eqref{15}, we have the following expression for $u^*$:
\begin{equation}\label{S2}
u^*=-S^{-1}\left(LX^*+\psi+\sum \limits_{i=1}^{l}\lambda_i\rho_i\right).
\end{equation}
We would like to mention that  $\Lambda\in \mathcal{L}^2_{\mathcal{F}}([s,T],\mathcal{S}^n)$ is in general unbounded and so $S^{-1}L$ might be unbounded, which leads to the fact that $u^*\notin \mathcal{L}^2_{\mathcal{F}}([s,T],\mathbb{R}^m)$. Therefore, \eqref{+9} only provides the open-loop solvability of Problem \ref{13}. It remains unclear if the  the framework of closed-loop solvability (in deterministic coefficient setting) is available to the SLQ problems with random coefficients \cite{sun2021indefinite}. Following \cite{zhang2023stochastic}, we make the following additional assumption for ensuring $u^*\in \mathcal{L}^2_{\mathcal{F}}([s,T],\mathbb{R}^m)$.
\begin{assumption}\label{bound}
$S^{-1}L\in\mathcal{L}^{\infty}_{\mathcal{F}}(\Omega,\mathcal{L}^2([s,T],\mathbb{R}^{m\times n}))$.
\end{assumption}
\begin{remark}
In the deterministic case when the coefficients $A,B,C,D,E,F,M$ are all deterministic and continuous functions, one has $P\in C([s,T],\mathcal{S}^n)$ and $\Lambda=0$ (\cite[Theorems 7.1 and 7.2]{Yong1999Stochastic}), and so Assumption \ref{bound} is naturally satisfied.
\end{remark}
We are now able to give the explicit formulation of the Lagranian dual functional  $\widetilde{J}(s,\xi,\lambda,u^*)$. To this end, we apply the It\^{o} formula to $\left\langle X^*,Y^*-Q-\sum \limits_{i=1}^{l}\lambda_iR_i\right\rangle$ and obtain that
\begin{align*}
&\quad\;\left\langle X^*(T),Y^*(T)-Q(T)-\sum \limits_{i=1}^{l}\lambda_iR_i(T)\right\rangle_{\mathcal{L}^2_{\mathcal{F}_{T}}(\Omega,\mathbb{R}^n)}-\left\langle \xi,Y^*(s)-Q(s)-\sum \limits_{i=1}^{l}\lambda_iR_i(s)\right\rangle_{\mathcal{L}^2_{\mathcal{F}_{s}}(\Omega,\mathbb{R}^n)}\\
&=\left\langle X^*,EX^*+Fu^*+\sum \limits_{i=1}^{l}\lambda_i\alpha_i+G+A^{\top}Q+C^{\top}\pi+G-A^{\top}Y^*-C^{\top}Z^*-L^{\top}S^{-1}\psi\right\rangle_{\mathcal{L}^2_{\mathcal{F}}([s,T],\mathbb{R}^n)}\\
&\quad\mbox{}+\left\langle X^*,\sum \limits_{i=1}^{l}\lambda_i\left(A^{\top}R_i+C^{\top}r_i+\alpha_i-L^{\top}S^{-1}\rho_i\right)\right\rangle_{\mathcal{L}^2_{\mathcal{F}}([s,T],\mathbb{R}^n)}\\
&\quad\mbox{}+\left\langle AX^*+Bu^*,Y^*-Q-\sum \limits_{i=1}^{l}\lambda_iR_i\right\rangle_{\mathcal{L}^2_{\mathcal{F}}([s,T],\mathbb{R}^n)}+
\left\langle CX^*+Du^*,Z^*-\pi-\sum \limits_{i=1}^{l}\lambda_ir_i\right\rangle_{\mathcal{L}^2_{\mathcal{F}}([s,T],\mathbb{R}^n)}\\
&=\left\langle X^*,EX^*+2G-L^{\top}S^{-1}\left(\psi+\sum \limits_{i=1}^{l}\lambda_i\rho_i\right)+
2\sum \limits_{i=1}^{l}\lambda_i\alpha_i\right\rangle_{\mathcal{L}^2_{\mathcal{F}}([s,T],\mathbb{R}^n)}+\left\langle X^*,Fu^*\right\rangle_{\mathcal{L}^2_{\mathcal{F}}([s,T],\mathbb{R}^n)}\\
&\quad\mbox{}+\left\langle u^*,B^{\top}Y^*+D^{\top}Z^*-B^{\top}Q-D^{\top}\pi-\sum \limits_{i=1}^{l}\lambda_i\left(B^{\top}R_i+D^{\top}r_i\right)\right\rangle_{\mathcal{L}^2_{\mathcal{F}}([s,T],\mathbb{R}^m)}\\
&=\left\langle X^*,EX^*+2Fu^*+2G\right\rangle_{\mathcal{L}^2_{\mathcal{F}}([s,T],\mathbb{R}^n)}+\left\langle u^*,Iu^*+2K\right\rangle_{\mathcal{L}^2_{\mathcal{F}}([s,T],\mathbb{R}^m)}\\
&\quad\mbox{}-\left\langle u^*+S^{-1}LX^*,\psi+\sum \limits_{i=1}^{l}\lambda_i\rho_i\right\rangle_{\mathcal{L}^2_{\mathcal{F}}([s,T],\mathbb{R}^m)}+2\sum \limits_{i=1}^{l}\lambda_i\left[\left\langle X^*,\alpha_i\right\rangle_{\mathcal{L}^2_{\mathcal{F}}([s,T],\mathbb{R}^n)}+\left\langle u^*,\beta_i\right\rangle_{\mathcal{L}^2_{\mathcal{F}}([s,T],\mathbb{R}^m)}\right],
\end{align*}
where we have used \eqref{S1}. Combining  \eqref{S2}, one has
\begin{align*}
\widetilde{J}(s,\xi,\lambda,u^*)
=&\left\langle \xi,P(s)\xi+2Q(s)+2\sum \limits_{i=1}^{l}\lambda_iR_i(s)\right\rangle_{\mathcal{L}^2_{\mathcal{F}_{s}}(\Omega,\mathbb{R}^n)}\\
&\mbox{}-\left\langle S^{-1}\left[\psi+\sum \limits_{i=1}^{l}\lambda_i\rho_i\right],\psi+\sum \limits_{i=1}^{l}\lambda_i\rho_i\right\rangle_{\mathcal{L}^2_{\mathcal{F}}([s,T],\mathbb{R}^m)}-2\langle\lambda,a\rangle\\
=&\left\langle \xi,P(s)\xi+2Q(s)\right\rangle_{\mathcal{L}^2_{\mathcal{F}_{s}}(\Omega,\mathbb{R}^n)}-\left\langle S^{-1}\psi,\psi\right\rangle_{\mathcal{L}^2_{\mathcal{F}}([s,T],\mathbb{R}^m)}+2\left\langle \mathbb{E}[R^{\top}(s)]\xi-a,\lambda\right\rangle\\
&\mbox{}-2\left\langle \mathbb{E}\left[\int_s^T\rho^{\top}(t)S^{-1}(t)\psi(t)dt\right],\lambda\right\rangle-\left\langle \mathbb{E}\left[\int_s^T\rho^{\top} (t) S^{-1}(t)\rho(t) dt\right]\lambda,\lambda\right\rangle,
\end{align*}
which is quadratic with respect to $\lambda\in \mathbb{R}^l$. Here $\rho=(\rho_1,\rho_2,\cdots,\rho_l)$ and $R=(R_1,R_2,\cdots,R_l)$. We summarize above arguments as the following theorem.
\begin{thm}\label{23}
Under Assumptions \ref{convexity} and \ref{bound}, the optimal control of Problem \ref{13} is
\begin{equation}\label{+9}
u^*=-S^{-1}\left(LX^*+\psi+\sum \limits_{i=1}^{l}\lambda_i\rho_i\right).
\end{equation}
Furthermore, the Lagrangian dual functional of Problem \ref{2} reads
\begin{align}
\widetilde{J}(s,\xi,\lambda,u^*)=&\left\langle \xi,P(s)\xi+2Q(s)\right\rangle_{\mathcal{L}^2_{\mathcal{F}_{s}}(\Omega,\mathbb{R}^n)}-\left\langle S^{-1}\psi,\psi\right\rangle_{\mathcal{L}^2_{\mathcal{F}}([s,T],\mathbb{R}^m)}+2\left\langle \mathbb{E}[R^{\top}(s)]\xi-a,\lambda\right\rangle\nonumber\\
&\mbox{}-2\left\langle \mathbb{E}\left[\int_s^T\rho^{\top}(t)S^{-1}(t)\psi(t)dt\right],\lambda\right\rangle-\left\langle \mathbb{E}\left[\int_s^T\rho^{\top} (t) S^{-1}(t)\rho(t) dt\right]\lambda,\lambda\right\rangle\label{12}.
\end{align}
\end{thm}
\begin{remark}
We would like to mention that  $\Lambda\in \mathcal{L}^2_{\mathcal{F}}([s,T],\mathcal{S}^n)$ is in general unbounded and so $S^{-1}L$ might be unbounded. Therefore, \eqref{+9} only provides the open-loop solvability of Problem \ref{13}. It remains unclear if the  the framework of closed-loop solvability (in deterministic coefficient setting) is available to the SLQ problems with random coefficients \cite{sun2021indefinite}.
\end{remark}
In the following, we would like to transform the constraints of $u$ into the ones of $\lambda$. Applying the It\^{o} formula to $\left\langle X^*,R_i\right\rangle$ derives
\begin{align}
\delta_i(\lambda)&=\langle X^*,\alpha_i\rangle_{\mathcal{L}^2_{\mathcal{F}}([s,T],\mathbb{R}^n)}+\langle u^*,\beta_i\rangle_{\mathcal{L}^2_{\mathcal{F}}([s,T],\mathbb{R}^m)}+\langle X^*(T),\gamma_i\rangle_{\mathcal{L}^2_{\mathcal{F}_{T}}(\Omega,\mathbb{R}^n)}\nonumber\\
&=\left\langle \xi,R_i(s)\right\rangle_{\mathcal{L}^2_{\mathcal{F}_{s}}(\Omega,\mathbb{R}^n)}-\left\langle S^{-1}\psi,\rho_i\right\rangle_{\mathcal{L}^2_{\mathcal{F}}([s,T],\mathbb{R}^m)}-\left\langle \rho^{\top}S^{-1}\rho_i,\lambda\right\rangle_{\mathcal{L}^2_{\mathcal{F}}([s,T],\mathbb{R}^l)}\label{19}
\end{align}
and so the dual problem can be specified as follows.
\begin{prob}\label{40}
Find $\lambda^*\in U_0$  such that
$$
\widetilde{J}(s,\xi,\lambda^*,u^{*}(\lambda^*))=\sup\limits_{\lambda\in U_0}\widetilde{J}(s,\xi,\lambda,u^*(\lambda)),
$$
where the Lagrangian dual functional $\widetilde{J}(s,\xi,\lambda,u^*(\lambda))$ is given by \eqref{12} and
$$
U_0=\left\{\lambda\in \mathbb{R}^l\big|\lambda_i\geq0\;(i=1,2,\cdots,l'),\; \delta_i(\lambda)\leq a_i \; (i=1,2,\cdots,l'),\; \delta_i(\lambda)=a_i\; (i=l'+1,l'+2,\cdots,l)\right\}.
$$
\end{prob}
\begin{remark}
Under Assumption \ref{17}, one can easily check that $U_0$ is nonempty, closed and convex in $\mathbb{R}^l$.
\end{remark}
We have the following strong duality result for Problem \ref{40}.
\begin{thm}\label{duality}
Under Assumptions \ref{17} and \ref{convexity}, there is
\begin{equation}\label{95}
\widetilde{J}(s,\xi,\lambda^*,u^{*}(\lambda^*))=J(s,\xi,\widehat{u}).
\end{equation}
\end{thm}
\begin{proof}
Consider the following two sets
\begin{align*}
U_5=\{&(\kappa,\tau)\in\mathbb{R}\times \mathbb{R}^l|\exists u\in\mathcal{L}^2_{\mathcal{F}}([s,T],\mathbb{R}^m),\;s.t.\;\left\langle u,\widetilde{\rho}_i\right\rangle_{\mathcal{L}^2_{\mathcal{F}}([s,T],\mathbb{R}^m)}-\widetilde{a}_i\leq\tau_i\; (i=1,2,\cdots,l'),\\
&\left\langle u,\widetilde{\rho}_i\right\rangle_{\mathcal{L}^2_{\mathcal{F}}([s,T],\mathbb{R}^m)}-\widetilde{a}_i=\tau_i (i=l'+1,l'+2,\cdots,l),\;\mbox{and}\;J(s,\xi,{u})-J(s,\xi,\widehat{u})\leq\kappa\},\\
U_6=\{&(\kappa,0^{l\times 1})\in\mathbb{R}\times \mathbb{R}^l,\,\kappa<0\}
\end{align*}
with $U_5\cap U_6=\emptyset$. According to Assumption \ref{convexity},  both $U_5$ and $U_6$ are convex. By the separation theorem of convex sets, there exists $(\overline{\kappa},\overline{\tau})\in\mathbb{R}\times \mathbb{R}^l$ such that $\inf\limits_{(\kappa,\tau)\in U_5}[\overline{\kappa}\kappa+\langle \overline{\tau},\tau\rangle]\geq \sup\limits_{\kappa<0}\overline{\kappa}\kappa$.
One can easily check that $\overline{\kappa}\geq0$, $\tau_i\geq0\; (i=1,2,\cdots,l')$ and $\tau_i\neq0\; (i=l'+1,l'+2,\cdots,l)$, and so
\begin{equation}\label{96}
\inf\limits_{(\kappa,\tau)\in U_5}[\overline{\kappa}\kappa+\langle \overline{\tau},\tau\rangle]\geq0.
\end{equation}
We claim that $\overline{\kappa}>0$. If $\overline{\kappa}=0$, then according to Assumptions \ref{17}, there exists a $\overline{u}\in\mathcal{L}^2_{\mathcal{F}}([s,T],\mathbb{R}^m)$ satisfying \eqref{99}. Therefore one has $\overline{\tau}_i=0\; (i=1,2,\cdots,l')$, where $\overline{\tau}=(\overline{\tau}_1,\overline{\tau}_2,\cdots,\overline{\tau}_l)$. And so
\begin{equation}\label{97}
\inf_{u\in\mathcal{L}^2_{\mathcal{F}}([s,T],\mathbb{R}^m)}\left[\left\langle u,\sum\limits_{i=l'+1}^{l}\overline{\tau}_i\widetilde{\rho}_i\right\rangle_{\mathcal{L}^2_{\mathcal{F}}([s,T],\mathbb{R}^m)}
-\sum\limits_{i=l'+1}^{l}\overline{\tau}_i\widetilde{a}_i\right]
\geq\inf\limits_{(\kappa,\tau)\in U_5}\sum\limits_{i=l'+1}^{l}\overline{\tau}_i\tau_i\geq0.
\end{equation}
By Remark \ref{98}, we know $\widetilde{\rho}=\sum\limits_{i=l'+1}^{l}\overline{\tau}_i\widetilde{\rho}_i\neq0$ and $\left\langle \overline{u},\widetilde{\rho}\right\rangle_{\mathcal{L}^2_{\mathcal{F}}([s,T],\mathbb{R}^m)}
=\sum\limits_{i=l'+1}^{l}\overline{\tau}_i\widetilde{a}_i$. Choosing $u_0=\overline{u}-\varepsilon\widetilde{\rho}$ for $\varepsilon>0$ small enough, then
$\left\langle u_0,\widetilde{\rho}\right\rangle_{\mathcal{L}^2_{\mathcal{F}}([s,T],\mathbb{R}^m)}
-\sum\limits_{i=l'+1}^{l}\overline{\tau}_i\widetilde{a}_i=-\varepsilon\|\widetilde{\rho}\|_{\mathcal{L}^2_{\mathcal{F}}([s,T],\mathbb{R}^m)}<0$, which contradicts \eqref{97}.

Letting $\overline{\lambda}=\frac{1}{\kappa}\overline{\tau}=(\overline{\lambda}_1,\overline{\lambda}_2,\cdots,\overline{\lambda}_l)$ ($\overline{\kappa}>0$) and combing \eqref{96}, we obtain
\begin{align*}
0&\leq\inf\limits_{(\kappa,\tau)\in U_5}[\kappa+\langle \overline{\lambda},\tau\rangle]\\
&\leq\inf_{u\in\mathcal{L}^2_{\mathcal{F}}([s,T],\mathbb{R}^m)}\left[J(s,\xi,{u})-J(s,\xi,\widehat{u})
+\sum\limits_{i=1}^{l}\overline{\lambda}_i\left[\left\langle u,\widetilde{\rho}_i\right\rangle_{\mathcal{L}^2_{\mathcal{F}}([s,T],\mathbb{R}^m)}
-\widetilde{a}_i\right]\right]\\
&=\widetilde{J}(s,\xi,\overline{\lambda},u^*(\overline{\lambda}))-J(s,\xi,\widehat{u})
\end{align*}
and so
$$
J(s,\xi,\widehat{u})\leq\widetilde{J}(s,\xi,\overline{\lambda},u^*(\overline{\lambda}))\leq\sup\limits_{\lambda\in U_0}\widetilde{J}(s,\xi,\lambda,u^*(\lambda)).
$$
On the other hand, since $\widetilde{J}(s,\xi,\lambda,u^*(\lambda))\leq J(s,\xi,\widehat{u})$ for all $\lambda\in U_0$, one has
$$
\widetilde{J}(s,\xi,\lambda^*,u^{*}(\lambda^*))=\sup\limits_{\lambda\in U_0}\widetilde{J}(s,\xi,\lambda,u^*(\lambda))\leq J(s,\xi,\widehat{u}),
$$
which ends the proof.
\end{proof}
According to Theorem \ref{duality}, we know that $u^{*}(\lambda^*)\in U^{ad}$ is an optimal control of Problem \ref{2}. Furthermore, if the invertibility assumption
\begin{equation}\label{26}
\mathbb{E}\left[\int_s^T\rho^{\top} (t) S^{-1}(t)\rho(t) dt\right]\in \mathcal{S}_{++}^{l}
\end{equation}
holds, then  the function $\widetilde{J}(s,\xi,\lambda,u^*(\lambda))$ is uniformly concave with respect to $\lambda$, which indicates the uniqueness of $\lambda^*=(\lambda^*_1,\lambda^*_2,\cdots,\lambda_l^*)^{\top}$. Therefore, we have the following theorem.
\begin{thm}\label{41}
Under Assumption \ref{17}, the solution $u^*(\lambda^*)$ to Problem \ref{40} is as well the solution to Problem \ref{2}. Moreover, if the invertibility assumption \eqref{26} holds, then
Problem \ref{40} is uniquely solvable.
\end{thm}
\begin{remark}
From Proposition \ref{27}, one has $S^{-1}(t)\in \mathcal{S}_{++}^{l}$ a.e. $t\in[s,T]$, $\mathbb{P}$-a.s.. Hence invertibility assumption \eqref{26} can be satisfied under some rank condition for $\rho$, which will be discussed in the next section.
\end{remark}
Finally, by standard arguments (see \cite{bonnans2013perturbation, jane2005necessary, guo2013mathematical}  and references therein), we are able to give the KKT condition for the optimal control of Problem \ref{2}, which shows the relationship between $\widehat{u}$ and $\lambda^*$.
\begin{thm}\label{42}
Under Assumptions \ref{17}, \ref{convexity} and \ref{bound}, $\widehat{u}$ is an optimal control of Problem \ref{2} if and only if the following KKT condition holds:
\begin{equation}\label{18}
\begin{cases}
\widehat{u}=-S^{-1}\left(L\widehat{X}+\psi+\rho^{\top}\lambda^*\right),\quad \lambda^*\geq 0\;(i=1,2,\cdots,l'),\\
\lambda^*_i\left(\delta_i(\lambda^*)- a_i\right)=0\;(i=1,2,\cdots,l'),
\quad
\delta_i(\lambda^*)= a_i\;(i=l'+1,l'+2,\cdots,l).
\end{cases}
\end{equation}
\end{thm}
\begin{remark}
If $\rho=0^{m\times l}$, then the term $\rho^{\top}\lambda^*$ can be removed and \eqref{18} reduces to $\widehat{u}=-S^{-1}\left(L\widehat{X}+\psi\right)$, i.e., $\widehat{u}$ is indeed the optimal control of the SLQ problem with no constraints. In other words, such an optimal control naturally satisfies all the affine constraints.
\end{remark}
\begin{remark}
If invertibility assumption \eqref{26} holds, then one can find a matrix $\rho_S\in \mathcal{S}_{++}^{l}$ with $\rho^2_S=\mathbb{E}\left[\int_s^T\rho^{\top} (t) S^{-1}(t)\rho(t) dt\right]$. Considering the bijective mapping $\overline{\rho}_S: U_0\rightarrow\overline{\rho}_S(U_0)\subset \mathbb{R}^l$ defined by $\overline{\rho}_S(\lambda)= {\rho}_S\lambda$, one can easily check that $\overline{\rho}_S(U_0)$ is nonempty, closed and convex, and so the optimal control of Problem \ref{2} can be equivalently given as follows:
\begin{equation*}
\begin{cases}
\widehat{u}=-S^{-1}\left(L\widehat{X}+\psi+\rho^{\top}\lambda^*\right),\\
\lambda^*=\rho^{-1}_S\mathop{\mbox{Proj}}_{{\overline{\rho}_S(U_0)}}\left\{\rho^{-1}_S\left(\mathbb{E}[R^{\top}(s)]\xi-a-\mathbb{E}\left[\int_s^T\rho^{\top}(t)S^{-1}(t)\psi(t)dt\right]\right)\right\},
\end{cases}
\end{equation*}
where $\mathop{\mbox{Proj}}_{{\overline{\rho}_S(U_0)}}[\cdot]$ represents the projection on the set ${\overline{\rho}_S(U_0)}$.
\end{remark}

\section{Sufficient Condition for Invertibility}
In this section, we provide a sufficient condition for ensuring the invertibility assumption \eqref{26}.

Recall the BSDE in the system \eqref{15} as follows:
\begin{equation}\label{28}
\begin{cases}
dR_i=-\left[A^{\top}R_i+C^{\top}r_i+\alpha_i-L^{\top}S^{-1}\rho_i\right]dt+r_idB(t), \quad t\in[s,T],\\
R_i(T)=\gamma_i.
\end{cases}
\end{equation}
We would like to find the expression of $\rho_i=\beta_i+B^{\top}R_i+D^{\top}r_i$ by $\alpha_i$, $\beta_i$ and $\gamma_i$ instead, since the BSDE \eqref{28} is hard to be solved directly in general. Consider the following system of SDEs:
\begin{equation}\label{35}
\begin{cases}
dV=V\left[A^{\top}-L^{\top}S^{-1}B^{\top}\right]dt+V\left[C^{\top}-L^{\top}S^{-1}D^{\top}\right]dB(t),\quad t\in[s,T],\\
dV_0=\left[(C^{\top}-L^{\top}S^{-1}D^{\top})^2-\left(A^{\top}-L^{\top}S^{-1}B^{\top})\right]\right]V_0dt-C^{\top}V_0dB(t),\quad t\in[s,T],\\
V(s)=V_0(s)=\mathcal{I}_n,
\end{cases}
\end{equation}
which admits a unique solution $(V,V_0)\in\mathcal{L}^2_{\mathcal{F}}(\Omega,C([s,T],\mathbb{R}^{n\times n}))\times\mathcal{L}^2_{\mathcal{F}}(\Omega,C([s,T],\mathbb{R}^{n\times n}))$. Applying the It\^{o} formula to $VV_0$, one has $V(t)V_0(t)=\mathcal{I}_n$ a.e. $t\in[s,T]$, $\mathbb{P}$-a.s.. Thus $V$ is invertible with $V^{-1}=V_0$.
Then, applying the It\^{o} formula to $VR_i$ derives
\begin{equation*}
\begin{cases}
d(VR_i)=-V\left[\alpha_i-L^{\top}S^{-1}\beta_i\right]dt
+V\left[(C^{\top}-L^{\top}S^{-1}D^{\top})R_i+r_i\right]dB(t),\\
V(T)R_i(T)=V(T)\gamma_i,
\end{cases}
\end{equation*}
and so $V(t)R_i(t)=\mathbb{E}\left[V(T)\gamma_i+\int_t^TV(\tau)\left[\alpha_i(\tau)-L^{\top}(\tau)S^{-1}(\tau)\beta_i(\tau)\right] d\tau\Big|\mathcal{F}_t\right]$. By means of Malliavin Calculus, we have the following theorems.
\begin{thm}\label{21}
If  $$\mathfrak{L}_t(\alpha_i,\beta_i,\gamma_i)=V(T)\gamma_i+\int_t^T V(\tau)\left[\alpha_i(\tau)-L^{\top}(\tau)S^{-1}(\tau)\beta_i(\tau)\right]d\tau$$
admits the Malliavin derivative $\mathfrak{D}_t\widetilde{\mathfrak{L}}_t(\alpha_i,\beta_i,\gamma_i)$ with $\mathbb{E}\left[\int_s^T|\mathfrak{D}_t\widetilde{\mathfrak{L}}_t(\alpha_i,\beta_i,\gamma_i)|^2dt\right]<\infty$, where $V$ is given by \eqref{35}. Then
$$
R_i=V_0\mathbb{E}\left[\mathfrak{L}_{\cdot}(\alpha_i,\beta_i,\gamma_i)\Big|\mathcal{F}_{\cdot}\right]
$$
and
\begin{equation}\label{111}
\rho_i=\beta_i+D^{\top}V_0\mathbb{E}\left[\mathfrak{D}_{\cdot}\mathfrak{L}_{\cdot}(\alpha_i,\beta_i,\gamma_i)\Big|\mathcal{F}_{\cdot}\right]
+(B^{\top}-D^{\top}C^{\top}+D^{\top}L^{\top}S^{-1}D^{\top})V_0\mathbb{E}\left[\mathfrak{L}_{\cdot}(\alpha_i,\beta_i,\gamma_i)\Big|\mathcal{F}_{\cdot}\right].
\end{equation}
\end{thm}
\begin{proof}
It is easy to derive the first equality. According to \cite[Remark 4.3]{nunno2009malliavin}, $\mathbb{E}\left[\mathfrak{D}_{\cdot}\mathfrak{L}_{\cdot}(\alpha_i,\beta_i,\gamma_i)\Big|\mathcal{F}_{\cdot}\right]$ admits a predictable modification. Therefore one can choose $\mathbb{E}\left[\mathfrak{D}_{\cdot}\widetilde{\mathfrak{L}}_{\cdot}(\alpha_i,\beta_i,\gamma_i)\Big|\mathcal{F}_{\cdot}\right]\in\mathcal{L}^2_{\mathcal{F}}([s,T],\mathbb{R}^n)$ . By Clark-Ocone's formula (see, e.g., \cite[Theorem 4.1]{nunno2009malliavin}), we see that
\begin{align*}
\mathfrak{L}_t(\alpha_i,\beta_i,\gamma_i)-\mathbb{E}\left[\mathfrak{L}_t(\alpha_i,\beta_i,\gamma_i)\Big|\mathcal{F}_t\right]
&=\int_t^TV(\tau)\left[(C^{\top}(\tau)-L^{\top}(\tau)S^{-1}(\tau)D^{\top}(\tau))R_i(\tau)+r_i(\tau)\right]dB_{\tau}\\
&=\int_t^T\mathbb{E}\left[\mathfrak{D}_{\tau}\mathfrak{L}_{\tau}(\alpha_i,\beta_i,\gamma_i)\Big|\mathcal{F}_{\tau}\right]dB_{\tau}.
\end{align*}
And so for $\tau\in[s,T]$, one has
$$
\mathbb{E}\left[\mathfrak{D}_{\tau}\mathfrak{L}_{\tau}(\alpha_i,\beta_i,\gamma_i)\Big|\mathcal{F}_{\tau}\right]
=V(\tau)\left[(C^{\top}(\tau)-L^{\top}(\tau)S^{-1}(\tau)D^{\top}(\tau))R_i(\tau)+r_i(\tau)\right].
$$
Therefore,
$$
r_i=V_0\mathbb{E}\left[\mathfrak{D}_{\cdot}\widetilde{\mathfrak{L}}_{\cdot}(\alpha_i,\beta_i,\gamma_i)\Big|\mathcal{F}_{\cdot}\right]-
(C^{\top}-L^{\top}S^{-1}D^{\top})V_0\mathbb{E}\left[\mathfrak{L}_{\cdot}(\alpha_i,\beta_i,\gamma_i)\Big|\mathcal{F}_{\cdot}\right]
$$
and second equality \eqref{111} follows immediately.
\end{proof}
\begin{thm}
The invertibility condition \eqref{26} holds if the vector group $\{\rho_i\}_{i=1}^{l}$ given by \eqref{111} is linearly independent.
\end{thm}
\begin{proof}
It follows from the fact that $S^{-1}(t)\in \mathcal{S}_{++}^{l}$ a.e. $t\in[s,T]$, $\mathbb{P}$-a.s. immediately.
\end{proof}
Similarly, we have the following result, by which Assumption \ref{17} can be more easily verified.
\begin{thm}
If  $$\widetilde{\mathfrak{L}}_t(\alpha_i,\gamma_i)=W(T)\gamma_i+\int_t^T W(\tau)\alpha_i(\tau)d\tau$$
admits the Malliavin derivative $\mathfrak{D}_t\widetilde{\mathfrak{L}}_t(\alpha_i,\beta_i,\gamma_i)$ with $\mathbb{E}\left[\int_s^T|\mathfrak{D}_t\widetilde{\mathfrak{L}}_t(\alpha_i,\beta_i,\gamma_i)|^2dt\right]<\infty$, where $W$ satisfies
\begin{equation*}
\begin{cases}
dW=WA^{\top}dt+WC^{\top}dB(t),\quad t\in[s,T],\\
W(s)=\mathcal{I}_n.
\end{cases}
\end{equation*}
Then
\begin{equation*}
\widetilde{\rho}_i=\beta_i+D^{\top}W^{-1}\mathbb{E}\left[\mathfrak{D}_{\cdot}\widetilde{\mathfrak{L}}_{\cdot}(\alpha_i,\gamma_i)\Big|\mathcal{F}_{\cdot}\right]
+(B^{\top}-D^{\top}C^{\top})W^{-1}\mathbb{E}\left[\widetilde{\mathfrak{L}}_{\cdot}(\alpha_i,\gamma_i)\Big|\mathcal{F}_{\cdot}\right].
\end{equation*}
\end{thm}
Two illustrative examples for Theorem \ref{21} are given as follows.
\begin{example}
Consider the controlled systems
\begin{equation*}
\begin{cases}
dX(t)=\frac{1}{2}X(t)dt+\left[X(t)+u(t)\right]dB(t), \quad t\in[0,1],\\
X(0)=\xi=(1,1)^{\top}.
\end{cases}
\end{equation*}
with the payoff functional
\begin{equation*}
J(0,(1,1)^{\top},u)=\mathbb{E}\left[\displaystyle{\int_0^1}\left[2\|X(t)\|^2+2\langle X(t), u(t)\rangle+\|u(t)\|^2\right]dt
\right],
\end{equation*}
where the controlled pair $(X,u)\in\mathcal{L}^2_{\mathcal{F}}([0,1],\mathbb{R}^2)\times\mathcal{L}^2_{\mathcal{F}}([0,1],\mathbb{R}^2)$ satisfies the following mixed equality and inequality constraints
$$
\left\langle X(1),e^{B(1)-1}\varsigma_1\right\rangle_{\mathcal{L}^2_{\mathcal{F}_{T}}(\Omega,\mathbb{R}^n)}\leq a_1,\quad \left \langle X(1),e^{B(1)-1}\varsigma_2\right\rangle_{\mathcal{L}^2_{\mathcal{F}_{T}}(\Omega,\mathbb{R}^n)}\leq a_2.
$$
with $\varsigma_1=(1,0)^{\top}$, $\varsigma_2=(0,1)^{\top}$. Find $u^*$ to minimize the payoff function is a case of Problem \ref{2}. We have from \eqref{4} and \eqref{15} that $P(t)=(e^{1-t}-1)\mathcal{I}_2$, $L(t)=S(t)=e^{1-t}\mathcal{I}_2$, $\Lambda(t)=0^{2\times2}$, $Q(t)=\pi(t)=0^{2\times 1}$. According to Theorem \ref{21}, one has $R_i(t)=r_i(t)=\rho_i(t)=e^{B(t)-t}\varsigma_i$ and $R_i(0)=\varsigma_i$ ($ i\in\{1,2\}$). Hence
$\delta_i=1-\lambda_i\langle e^{1-t}\rho_i,\rho_i\rangle_{\mathcal{L}^2_{\mathcal{F}}([s,T],\mathbb{R}^m)}=1-(e-1)\lambda_i.$

Next, by the KKT condition \eqref{18}, we have
\begin{equation*}
\begin{cases}
\widehat{u}(t)=-\widehat{X}(t)-e^{B(t)-t}\lambda^*,\quad\lambda^*_1\geq0,\quad\lambda^*_1\geq\frac{1-a_1}{2}\\
\lambda^*_1(1-\lambda^*_1)=0,\quad 1-(e-1)\lambda^*_2=a_2
\end{cases}
\end{equation*}
and so  $\lambda^*=(\max\left\{0,\frac{1-a_1}{e-1}\right\},\frac{1-a_2}{e-1})^{\top}$. Finally, by $\mathbb{E}\left[\int_0^1\rho^{\top} (t) S^{-1}(t)\rho(t) dt\right]=(e-1)\mathcal{I}_2\in \mathcal{S}_{++}^2$ and Theorems \ref{41} and \ref{42}, the unique feedback optimal control reads
$$
\widehat{u}(t)=-\widehat{X}(t)-e^{B(t)-t}
\begin{bmatrix}
\max\left\{0,\frac{1-a_1}{e-1}\right\}\\
\frac{1-a_2}{e-1}
\end{bmatrix},\quad t\in[0,1]
$$
and the value function is
$
\widetilde{J}(0,(1,1)^{\top},\lambda^*,u^{*}(\lambda^*))=2(e-1)+\langle2\xi-2a-(e-1)\lambda,\lambda\rangle.
$
\begin{figure}[htbp]
\centering
 %Requires \usepackage{graphicx}
\includegraphics[height=7cm, width=14cm]{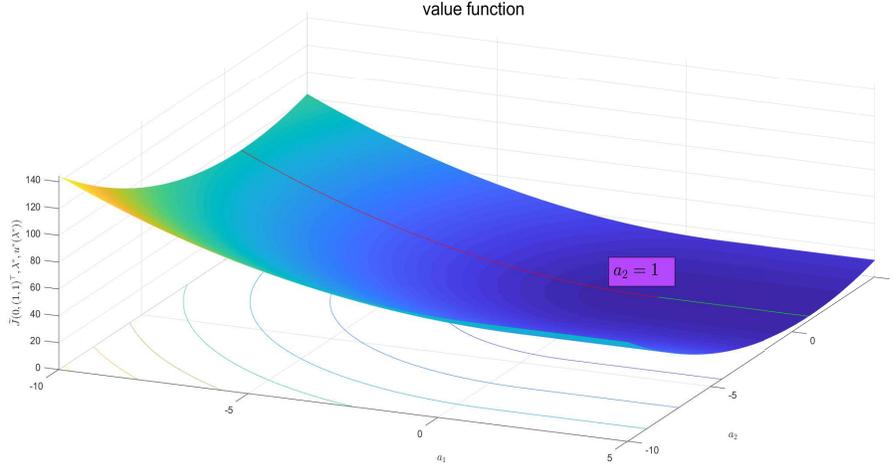}
\caption{Change in $\widetilde{J}(0,(1,1)^{\top},\lambda^*,u^{*}(\lambda^*))$ with $a_{1}$ and $a_2$.}
\label{Fig.1}
\end{figure}
\end{example}

Figure 1 shows that: (i) for fixed $a_1$, the value function $\widetilde{J}(0,(1,1)^{\top},\lambda^*,u^{*}(\lambda^*))$ achieves its minimum when $a_2=1$, (ii) for fixed $a_2$, the value function $\widetilde{J}(0,(1,1)^{\top},\lambda^*,u^{*}(\lambda^*))$ is gradually decreasing with the growth of $a_1$. This indicates that strengthening the constraints results in a growth of the cost.

\begin{example}
Consider the controlled systems
\begin{equation*}
\begin{cases}
dX(t)=[-\frac{1}{2}X(t)+u(t)]dt+X(t)dB(t), \quad t\in[0,1],\\
X(0)=\xi=(1,1)^{\top}.
\end{cases}
\end{equation*}
with the payoff functional
\begin{equation*}
J(0,(1,1)^{\top},u)=\mathbb{E}\left[\displaystyle{\int_0^1}\|u(t)\|^2dt\right]+\mathbb{E}\left[\|X(1)\|^2+2\langle X(1),2e^{B(1)-1}\xi\rangle
\right],
\end{equation*}
where the controlled pair $(X,u)\in\mathcal{L}^2_{\mathcal{F}}([0,1],\mathbb{R}^2)\times\mathcal{L}^2_{\mathcal{F}}([0,1],\mathbb{R}^2)$ satisfies the following affine constraints
$$
\left\langle X(T),2e^{B(1)-1}\varsigma_1\right\rangle_{\mathcal{L}^2_{\mathcal{F}_{T}}(\Omega,\mathbb{R}^n)}\leq a_1,\quad \left \langle X(T),2e^{B(1)-1}\varsigma_2\right\rangle_{\mathcal{L}^2_{\mathcal{F}_{T}}(\Omega,\mathbb{R}^n)}\leq a_2.
$$
with $\varsigma_1=(1,0)^{\top}$, $\varsigma_2=(0,1)^{\top}$. Now we would like to find the optimal control $u^*$ to minimize the payoff functional. It follows from \eqref{4} and \eqref{15} that $P(t)=L(t)=\frac{1}{2-t}\mathcal{I}_2$, $S(t)=\mathcal{I}_2$, $\Lambda(t)=0^{2\times2}$. According to Theorem \ref{21}, one has $R_i(t)=r_i(t)=\rho_i(t)=\frac{2}{2-t}e^{B(t)-t}\varsigma_i$ and $R_i(0)=\varsigma_i$. Similarly, we have $Q(t)=\pi(t)=\frac{2}{2-t}e^{B(t)-t}\xi$. Hence
$\delta_i=1-\left\langle Q,\rho_i\right\rangle_{\mathcal{L}^2_{\mathcal{F}}([s,T],\mathbb{R}^m)}-\lambda_i\langle\rho_i,\rho_i\rangle_{\mathcal{L}^2_{\mathcal{F}}([s,T],\mathbb{R}^m)}=-1-2\lambda_i.$

Next, by the KKT condition \eqref{18}, we have
\begin{equation*}
\begin{cases}
\widehat{u}(t)=-\frac{1}{2-t}(\widehat{X}(t)+2e^{B(t)-t}\lambda^*),\quad\lambda^*_1\geq 0,\quad\lambda^*_1\geq-\frac{1+a_1}{2},\quad\lambda^*_2\geq 0,\quad\lambda^*_2\geq-\frac{1+a_2}{2},\\
\lambda^*_1(1+2\lambda^*_1+a_1)=0,\quad \lambda^*_2(1+2\lambda^*_2+a_2)=0
\end{cases}
\end{equation*}
and so $\lambda^*=(\max\left\{0,-\frac{1+a_1}{2}\right\},\max\left\{0,-\frac{1+a_2}{2}\right\})^{\top}$. Finally, by $\mathbb{E}\left[\int_0^1\rho^{\top} (t) S^{-1}(t)\rho(t) dt\right]=2\mathcal{I}_2\in \mathcal{S}_{++}^2$, $|\frac{1}{2-t}|\leq1$, and Theorems \ref{41} and \ref{42},  the unique feedback optimal control reads
$$
\widehat{u}(t)=-\frac{1}{2-t}\left(\widehat{X}(t)+ e^{B(t)-t}
\begin{bmatrix}
\max\left\{2,1-a_1\right\}\\
\max\left\{2,1-a_2\right\}
\end{bmatrix}\right),\quad t\in[0,1]
$$
and the value function is
$
\widetilde{J}(0,(1,1)^{\top},\lambda^*,u^{*}(\lambda^*))=4-2\langle a+\xi+\lambda^*,\lambda^*\rangle.
$
\begin{figure}[htbp]
\centering
 %Requires \usepackage{graphicx}
\includegraphics[height=7cm, width=14cm]{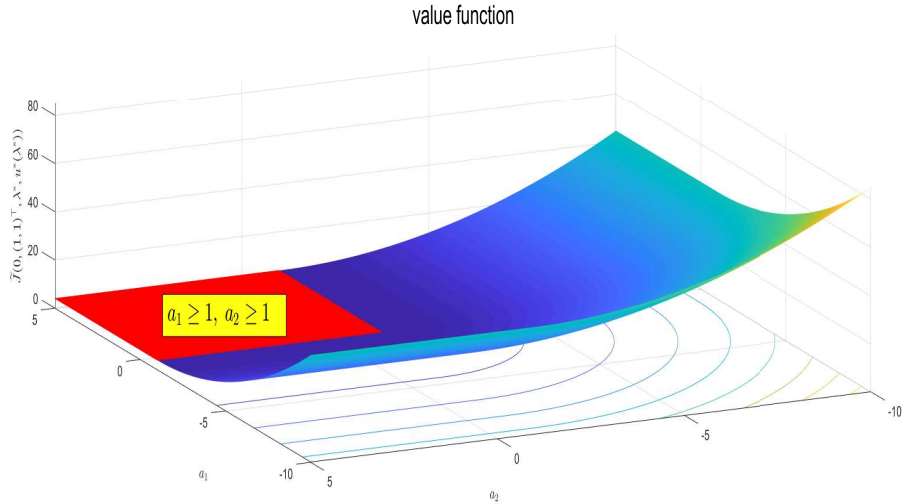}
\caption{Change in $\widetilde{J}(0,(1,1)^{\top},\lambda^*,u^{*}(\lambda^*))$ with $a_{1}$ and $a_2$.}
\label{Fig.2}
\end{figure}

Figure 2 shows that the value function $\widetilde{J}(0,(1,1)^{\top},\lambda^*,u^{*}(\lambda^*))$ is gradually reducing with the growth of both $a_1$ and $a_2$, and achieves its minimum when $a_1\geq1,\;a_2\geq1$, which indicates that the reinforcement of the constraints results in an increasing of the cost.
\end{example}

\section{Conclusions}
This paper is devoted to studying the SLQ problems with affine constraints in random coefficients case. By solving the relaxed SLQ problem, the dual problem of original problem is obtained and the state feedback form of the open-loop optimal control is given for the dual problem. Then, the strong duality of the dual problem is obtained under the Slater condition and the invertibility assumption is introduced for ensuring the uniqueness of solutions to the dual problem. Finally, the KKT condition is established  for solving the original problem and a sufficient condition is provided for guaranteeing the invertibility assumption. Two examples are provided to illustrate the effectiveness of the main results.

It is worth mentioning extension of our problem formulation, including the SLQ problems with delay/jumps/partial information \cite{li2020linear, moon2021indefinite, zhang2021linear}, promises to be interesting and important topics. Moreover, the stochastic Nash/Stackelberg differential games with affine constraints call for further research. We plan to address these problems as we continue our research.

\section*{Acknowledgements}
The authors are grateful to  Professor Qi L\"{u} for his constructive comments, which helps us to improve the paper.

%\end{CJK}
\end{document}